\documentclass[11pt]{amsart}

\usepackage{amsmath}
\usepackage{amsthm}
\usepackage{amssymb}
\usepackage{amscd}
\usepackage{bbold}
\usepackage[pdftex]{graphicx}
\usepackage{enumerate}

\newtheorem{lem}{Lemma}[section]
\newtheorem{thm}[lem]{Theorem}
\newtheorem{prop}[lem]{Proposition}

\theoremstyle{definition}

\newtheorem{rem}[lem]{Remark}

\def\benm{\begin{enumerate}}            
\def\eenm{\end{enumerate}}              

\title[Frames that are robust to erasure]
{Signal recovery \\ and frames that are robust to erasure}

\date{}

\subjclass[]{}

\author{Enrico Au-Yeung\\ \quad University of British Columbia, Canada}
\address{Department of Mathematics\\
          University of British Columbia\\
         Vancouver, B.C. \\
         Canada V6T 1Z2}
 \email{enricoauy@math.ubc.ca}

\begin{document}

\begin{abstract}
We consider finite frames with high redundancy so that if half the terms transmitted from the sender
are randomly deleted during transmission, then on average, the receiver can still recover the signal to within a high level of accuracy. 
This follows from a result in random matrix theory. 
We also give an application of the operator Khintchine inequality in the setting of signal recovery when the signal is a matrix with a sparse representation.

\end{abstract}

\maketitle

\section{Introduction and background}

In a finite-dimensional Hilbert space $\mathcal{H}$, any spanning set of vectors is a frame for $\mathcal{H}$.  A finite set of vectors $\{f_n\}_{n=1}^{N}$ is a tight frame for $\mathcal{H}$ if and only if there is a positive constant $\alpha$, such that every element $x$  in $\mathcal{H}$ can be written as $x = \alpha \sum_{n=1}^{N} \langle x, f_n \rangle f_n$.  Suppose the following communication protocol has been established between Alice and Bob.  We assume that a message is represented by an element $x \in \mathcal{H}$.  When Alice wants to send a message to Bob, she sends the $N$ numbers $\{\langle x, f_n\rangle\}_{n=1}^{N}$.  Once Bob receives those $N$ numbers, he reconstructs the message $x$ by using the formula  $x = \alpha \sum_{n=1}^{N} \langle x, f_n \rangle f_n$.  If the communication channel is perfect, so that no transmission error can occur, then the message can be perfectly reconstructed. 
But in a more realistic setting where the channel is not perfect, some information may be lost during the transmission.  This situation is explored further in section $2$ of this article.\\

The flexibility of using a redundant set of vectors instead of a basis to reconstruct a vector suggests the following scenario.  An adversary removes a proportion $p$ of the $N$ transmitted inner products, an approximation of the original signal is constructed from the remaining $(1-p)N$ inner products. It is desirable to design finite frames that are robust to erasures.
Finite tight frames that are robust to erasure have been considered in \cite{Casazza2003}, \cite{GoyalKovacevic2001},  and \cite{Paulsen2004}.
Examples of tight frames include the Grassmannian frames \cite{Strohmer2003} and equiangular tight frames constructed from the Singer difference set \cite{XiaZhou2005}. \\


 Fickus and Mixon \cite{Fickus2012} introduced the following concept.
Given $p \in [0,1]$, and $C \geq 1$, an $M$ by $N$ frame $F$ is $(p,C)$-numerically erasure-robust
if for every subset $\Omega \subseteq \{1, \ldots, N\}$ of size $K \equiv (1-p)N$, the corresponding $M$ by $K$ submatrix
$F_{\Omega}$ has condition number Cond($F_{\Omega}$) less than or equal to $C$. They proved that a certain type of tight frame are
numerically robust.

\begin{thm}\label{thm:Fickus} Take $M = q + 1$ and $N = q^2 + q + 1$ for some prime power $q$ and let $F$ be the $M$ by $N$ equiangular tight frame
from the $(N,M,1)$-Singer difference set, as in \cite{XiaZhou2005}.  Then $F$ is a $(p,C)$-numericaly erasure-robust frame 
for every $p \leq \frac{1}{2} - \frac{C^2}{C^4 + 1}$.
\end{thm}

This article is partly inspired by the above theorem.  The presence of the factor 1/2 in the above theorem is intriguing and as Fickus and Mixon pointed out, this threshold of one half 
seems to be not just an artifact of the proof.

For a given vector $z \in \mathbb{R}^n$, define the rank-one operator $z \otimes z$  by $(z \otimes z)(x) = \langle z, x \rangle z$ for all $x \in \mathbb{R}^n$.  Note that $x = \sum_{n=1}^{N} \langle  z_n, x \rangle z_n$ for all $x \in \mathbb{R}^n$ if and only if the identity operator on $\mathbb{R}^n$ can be written as a sum of rank-one operators,
\[ I_n = \sum_{n=1}^{N} z_n \otimes z_n. \]

Motivated by Theorem \ref{thm:Fickus}, we consider the following scenario. Suppose an adversary 
randomly removes one half of the transmitted inner products.  If the frame has enough redundancy, then on average, 
can the signal be nearly recovered?  
The answer is positive and it follows as a consequence of the following
 important result of Rudelson \cite{Rudelson1995}.

\begin{thm}\label{thm:Rudelson}
Let $z_1,  \ldots, z_M$ be vectors in $\mathbb{R}^{n}$ and let $\epsilon_1, \ldots \epsilon_{M}$ be independent  Bernoulli variables
 taking values 1, -1 with probability 1/2.  There is a constant $C > 0$, such that
\[E \| \sum_{i = 1}^{M} \epsilon_i z_i \otimes z_i \| \leq C \sqrt{\log n} \cdot \max_{1 \leq i \leq M} \| z_{i} \| \cdot \| \sum_{i = 1}^{M} z_i \otimes z_i \|^{1/2}.  \]
\end{thm}
For an elegant proof, see \cite{Roberto2010}.  
 The original proof of Rudelson's inequality as it appears in \cite{Rudelson1995} uses the non-commutative Khintchine inequality of Lust-Piquard and Pisier \cite{LustPiquardPisier1991}, see also \cite{Buchholz2001}, \cite{Haagerup2007}.  Before stating that theorem, recall that for $p \geq 1$,  the Schatten-class norms of an $N$ by $N$ matrix A are defined by $\| A \|_{\mathcal{C}_{p}} = ( \sum_{j=1}^{N} | \sigma_{j}(A) |^{p})^{1/p}$, where $\sigma_j(A) $ are the singular values of $A$.  
\begin{thm}\label{thm:Khintchine}
Let $(A_j)_{j=1}^{n}$ be a sequence of matrices of the same dimension and let $(\epsilon_j)_{j=1}^{n}$ be a sequence of independent Bernoulli random 
variables.  For any positive integer $m$,
\[ \left[ E \left\| \sum_{j=1}^{n} \epsilon_j A_j \right\|_{\mathcal{C}_{2m}}^{2m} \right]^{1/(2m)} \leq C_m \max \left( \left\| \left( \sum_{j=1}^{n} A_{j}^{\ast} A_j \right)^{1/2}\right\|_{\mathcal{C}_{2m}}, \ \left\| \left( \sum_{j=1}^{n} A_j A_{j}^{\ast}  \right)^{1/2}\right\|_{\mathcal{C}_{2m}}  \right) \]
where the constant $C_m$ is given by \[ C_m = 2 \left(\frac{(2m)!}{2^m \ m!}\right)^{1/(2m)}. \]
\end{thm}

An application of Theorem \ref{thm:Khintchine} will be given in section 3.  Throughout this article, $\{\epsilon_i\}$ are independent Bernoulli random variables as in Theorem \ref{thm:Rudelson}.  The notation $E_{\epsilon}$ is used to emphasize that the expected value of a quantity is with respect to the random variables $\{\epsilon_i\}$.

\section{Frames that are robust to erasure}
For any $x \in \mathbb{R}^n,$ we can write $x = \frac{1}{n} \sum_{j=1}^{n} \langle x,  \sqrt{n} e_j \rangle \sqrt{n} e_j$, where  $\{e_i\}_{i=1}^{n}$ is the standard orthonormal basis of $\mathbb{R}^n.$
Thus $\{ \sqrt{n} e_j\}_{j=1}^n$ is a tight frame for $\mathbb{R}^n$.  This is an example of a large class of tight frames.
Let $z_1, z_2, \ldots z_M$ be vectors in $\mathbb{R}^n$ such that $\| z_j \|_{2}^2 = n$ for each $j$, and for all $x \in \mathbb{R}^n, x = \frac{1}{M} \sum_{j=1}^{M} \langle z_j, x \rangle z_j.$
 When the number $M$ is a lot larger than $n$, then the tight frame $\{z_j\}_{j=1}^M$ is highly redundant.
Let $\theta_1, \theta_2, \ldots \theta_M$ be independent random variables such that Pr$(\theta_j = 1) = 1/2 =$ Pr$(\theta_j = 0)$ for $1 \leq j \leq M$.
The inner products $\{ \langle z_n, x \rangle \colon 1 \leq n \leq M\}$ are sent, but some of them are randomly erased. The transmitted coefficients are represented by $\{ \langle z_n, x \rangle \colon n \in \Omega \}.$
  To construct an approximation of $x$ from the transmitted inner products, we compute
\[y = \frac{2}{M} \sum_{j \in \Omega} \langle z_j, x \rangle z_j, \quad \text{ where } \Omega = \{j \colon \theta_j = 1\}.\]
The next theorem shows that the approximation error $\| x - y \|_2$ can be made arbitrarily small, provided $M$ is sufficiently large.  
\begin{thm}
Let $z_1, z_2, \ldots z_M$ be vectors in $\mathbb{R}^n$ such that for each $x \in \mathbb{R}^n$, 
\[
x = \frac{1}{M} \sum_{j=1}^{M} \langle z_j, x \rangle z_j.
\]
Suppose  $\| z_j \|_{2}^2 = n$ for each $j$.
Let $\theta_1, \theta_2, \ldots \theta_M$ be independent random variables taking values 1 or 0 with probability 1/2.
Let $\Omega = \{j \colon \theta_j = 1\}.$ 
Given $x \in \mathbb{R}^n$, define the random vector $y$ by
\[ y = \frac{2}{M} \sum_{j \in \Omega}  \langle z_j, x \rangle z_j. \]
Let $\epsilon$ be any positive number satisfying $\epsilon^2 \geq \frac{1}{M} \ n \log n$.\\

Then there is an absolute constant $C > 0$ such that
\[E \| x - y \|_2 \leq C \ \epsilon \ \| x \|_2. \]
\end{thm}

\begin{proof}
Note that the approximation error $y - x$ can be expressed as
\[ y - x = \frac{2}{M} \sum_{j=1}^{M}  \theta_j \langle z_j, x \rangle z_j - \frac{1}{M}  \sum_{j=1}^{M}  \langle z_j, x \rangle z_j = \frac{1}{M} \sum_{j=1}^{M} \epsilon_j \langle z_j, x \rangle z_j \]
where $\{\epsilon_j\}_{j=1}^M$ are Bernoulli random variables.
Therefore, it is enough to prove that 
\[E \| \frac{1}{M} \sum_{j=1}^{M} \epsilon_{j} z_j \otimes z_j \| \leq C \ \epsilon.\]
By Rudelson's inequality (Theorem \ref{thm:Rudelson}), there exists a positive constant $C$, such that
\begin{equation}\label{eqn:Rudelson}
E \| \sum_{j = 1}^{M} \epsilon_j z_j \otimes z_j \| \leq C \sqrt{\log n} \cdot m_1 \cdot \| \sum_{j = 1}^{M} z_j \otimes z_j \|^{1/2}. 
\end{equation}
Here, $m_1 = \max_{j} \| z_j \| = \sqrt{n}.$ 
By hypothesis, $I = \frac{1}{M} \sum_{j=1}^{M} z_j \otimes z_j$, which means 
\begin{equation}\label{eqn:tight_frame_sqrt}
\| \sum_{j=1}^{M} z_j \otimes z_{j} \|^{1/2} = \sqrt{M}. 
\end{equation}
By equations (\ref{eqn:Rudelson}) and (\ref{eqn:tight_frame_sqrt}), we have
\begin{equation*}
E \|  \sum_{j=1}^{M} \epsilon_{j} z_j \otimes z_j \| \leq \sqrt{M} \ C \sqrt{\log n} \ \sqrt{n} .
\end{equation*}
After dividing  both sides by $M$, and since $\epsilon^2 \geq \frac{n}{M} \log n$,  we obtain
\[ E \| \frac{1}{M} \sum_{j=1}^{M} \epsilon_{j} z_j \otimes z_j \| \leq C \ \epsilon.\]

\end{proof}
The theorem above states that on average we can nearly recover the signal, provided the frame has enough redundant elements.   
 \section{An application of the operator Khintchine inequality}
 
We now give an interesting application of the operator Khintchine inequality (Theorem \ref{thm:Khintchine})  in the setting of signal recovery.  
Data acquisition is expensive in electromagnetic tomography, 
when data are gathered by receivers at multiple positions from multiple sources \cite{DornMiller2000}.  The total number of receivers, as well as their positions, might vary with the source.  When the signal is a matrix and it 
is prohibitively expensive to measure this signal directly, one alternative is to use the probing method  (see \cite{ChanMathew1992}).  The basic idea  is 
 to approximate a matrix by a matrix having a specified sparsity pattern, using a few matrix-vector products with carefully chosen probe vectors.\\

Let $U_1, U_2, \ldots U_n$ be a fixed set of n by n matrices.  These matrices are assumed to be known.  
Suppose $A$ is a matrix that lies in the span of $\{U_j\}_{j=1}^n$.  
To determine $A$, it is sufficient to determine the coefficients $\{\lambda_j\}_{j=1}^n$ so that
\begin{equation}\label{eqn:sparse1}
A = \lambda_1 U_1 + \lambda_2 U_2 + \ldots + \lambda_n U_n.
\end{equation}
 We say that the matrix $A$ has a sparse representation because only $n$ numbers are needed to determine the matrix $A$, even though the ambient dimension of the matrix $A$ is $n^2$.  One way to determine the numbers $\{\lambda_j\}_{j=1}^n$ is by writing all the matrices as vectors in ${n^2}$  dimension, and apply the method of least squares to solve for the $n$ coefficients.  An alternative approach is to treat the problem as an $n$-dimensional problem, even though all the matrices are $n^2$-dimensional objects. \\

 Fix any $x \in \mathbb{R}^n$.  Mutiply each side of equation (\ref{eqn:sparse1}) by $x$, we obtain
 \begin{equation}\label{eqn:sparse2}
 A x = \lambda_1 U_1 x + \lambda_2 U_2 x + \ldots + \lambda_n U_n x.
\end{equation}
If we  let $u_1 = U_1x, u_2 = U_2 x, \ldots u_n = U_n x$, then equation (\ref{eqn:sparse2}) can be written as
\begin{equation*}
A x = [u_1 | u_2 | \ldots | u_n] \ \lambda
\end{equation*}
where the matrix $D = [u_1 | u_2 | \ldots | u_n]$ is the n by n matrix whose column $j$ is $u_j$ for $1 \leq j \leq n$,  and $\lambda$ is the column vector $\lambda = (\lambda_1, \ldots, \lambda_n)^{T}$.  The matrix $D$ is called the dictionary for the class of signals that are spanned by the columns of $D$.  Given any $x \in \mathbb{R}^n$,  suppose after  $y = Ax$ is observed, the matrix $A$ is discarded.  Then from $y = D \lambda$, we can solve for $\lambda = D^{-1} y$ to recover the matrix $A$, at least when the matrix $D$ is invertible. The numerical stability of the solution for $\lambda$ depends on 
the condition number of the matrix $D$.\\

If $x \in \mathbb{R}^n$ is a randomly selected vector, the method just described can still be used to recover the matrix $A$, and the dictionary matrix $D$ will be a random matrix.   Consider the expected value of $\| D - E(D) \|,$ which measures on average, how much the random matrix $D$ deviates from its expected value.
Suppose $x = (x_1, x_2, \ldots, x_n)$, where all $x_j$ are independent random variables from the same probability distribution.  
For simplicity, assume each $ | x_j | \leq 1$. The matrix $D$ depends on the random vector $x$ and can be expressed as
$D = \sum_{j=1}^{n} x_j T_j$. Each $T_j$ is an $n$ by $n$ matrix.\\

Indeed, for each $j$, the $j$-th column of matrix $T_k$ is the $k$-th column of matrix $U_j$.
To be precise, write $U_j = [ u_{1}^{(j)} | u_{2}^{(j)} | \ldots |  u_{n}^{(j)} ]$ for the columns of matrix $U_j$.  
Then $T_k = [u_{k}^{(1)} | u_{k}^{(2)} | \ldots | u_{k}^{(n)} ].$ \\


Let $E_1 = E( \| D - E(D) \| )$. For each $1 \leq j \leq M,$ let $\widetilde{x}_j$ be a random copy of $x_j$, i.e. $\widetilde{x}_j$ and $x$ are independent random variables with the same  distributions. 
Note that $x_k - \widetilde{x}_k$ is a symmetric random variable and so it has the same
probability distribution as $\epsilon_k (x_k - \widetilde{x}_k)$.
\begin{align*}
E_{1} = E_{x} \|  \sum_{j=1}^n x_j T_j - E_{\widetilde{x}}(\sum_{j=1}^n \widetilde{x}_j T_j  ) \|  & \leq E_{x} E_{\widetilde{x}}  \|  \sum_{j=1}^n x_j T_j -  \sum_{j=1}^n \widetilde{x}_j T_j  \| \quad (\text{by convexity}) \\
&  = E_{x} E_{\widetilde{x}}  E_{\epsilon} \|  \sum_{j=1}^n \epsilon_j (x_j - \widetilde{x}_j) T_j  \| \quad (x_j  - \widetilde{x}_j  \text{ is symmetric}) \\
&  \leq 2 E_{x} E_{\epsilon}  \|  \sum_{j=1}^n \epsilon_{j} x_j T_j  \| .
\end{align*}


We now invoke the Contraction Principle (\cite{Talagrand1991}, Ch. 4),
\begin{prop}\label{prop:Contraction}
Let $f_1, f_2, \ldots, f_n$ be any finite sequence of elements in a Banach space.
For any choice of real numbers $x_1, x_2, \ldots, x_n$ where $ | x_j | \leq b $ for each j, 
\[ E_{\epsilon} \left\| \sum_{k=1}^{n} \epsilon_{k} x_{k} f_{k} \right\| \leq  b \ E_{\epsilon} \left\| \sum_{k=1}^{n} \epsilon_{k}  f_{k} \right\|. \]
\end{prop}
Combining the above calculations with Proposition (\ref{prop:Contraction}), we see that
\[ E \left( \| D - E(D) \| \right) \leq 2  E_{\epsilon}  \|  \sum_{k=1}^n \epsilon_{k} T_k  \| .\]
It remains to find an upper bound for $E \| \sum_{k=1}^n \epsilon_k T_k \|$. 
 Let $Z = \sum_{k=1}^n \epsilon_k T_k.$  
 The operator norm of $Z$ is bounded by the Schatten-class norm,  
  $\| Z \|_{\mathcal{C}_{p}} = ( \sum_{j=1}^{N} | \sigma_{j}(Z) |^{p})^{1/p}$.  Additional information on the structure of matrices $T_k$ will allow us to use Theorem $\ref{thm:Khintchine}$ to bound the norm of $Z$. \\
 
 Suppose each matrix $T_k$  satisfies $T_k^{\ast} T_k = \frac{1}{n} I_n$, where $I_n$ is the identity on $\mathbb{R}^n$. (See Remark \ref{rem:final_remark})
 Then $ \sum_{k=1}^n T_k^{\ast} T_k  = I_n $
 and hence
 \begin{equation}\label{eqn:Tk}
  \left\| \left( \sum_{k=1}^n T_k^{\ast}T_k \right)^{1/2} \right\|_{\mathcal{C}^p}^{p}   = n.
 \end{equation}
 Using equation (\ref{eqn:Tk}), together with  Theorem \ref{thm:Khintchine}, we obtain 
 \[ E \left( \| Z \|^{2m} \right) \leq E \left( \| Z \|_{\mathcal{C}_{2m}}^{2m} \right) \leq 2^{2m} \cdot \frac{(2m)!}{2^m \ m!} \cdot n \]
 for a large $m$.  By Stirling approximation,
 \[ \frac{(2m)!}{2^m \ m!} \leq \sqrt{2} \left( \frac{2}{e} \right)^{m} \cdot m^m. \]
 Combining the above estimates, we see that
 \[ E \left(\| Z \|^{2m} \right) \leq 2^{2m} \sqrt{2} \left( \frac{2}{e} \right)^{m} \cdot m^m \cdot n. \]
 Since $(E \| Z \|)^{2m} \leq  E \left( \| Z \|^{2m} \right)$, this implies there exists a constant $\alpha_0$, such that
 \[ E \| Z \| \leq \alpha_0 \sqrt{ m} \ n^{1/(2m)}.\]
 Finally, by choosing $2m = \log n$,  we conclude that
 \[ E \| Z \| \leq \alpha_0 \ e \ \sqrt{ \log n}. \]
 Thus, we have proven:
 
\begin{thm}
Let $x_1, x_2, \ldots , x_n$ be independent and identically distributed random variables, where each $| x_j | \leq 1$.  
Let $T_1, T_2, \ldots, T_n$ be a fixed set of $n$ by $n$ matrices.
Let $D = \sum_{k=1}^n x_k T_k$ be a random matrix.
Suppose each matrix $T_k$  satisfies $T_k^{\ast} T_k = \frac{1}{n} I_n$, where $I_n$ is the identity on $\mathbb{R}^n$ .
There exists a constant $\alpha_1$ such that
\[ E \left( \| D - E(D) \| \right) \leq \alpha_1  \sqrt{ \log n}. \]
 \end{thm}



 



\begin{rem}\label{rem:final_remark}
The assumption that $T_k^{\ast} T_k = \frac{1}{n} I_n$ means the columns of $T_k$ are orthogonal. Since matrix $A$ satisfies $A = \sum_{j=1}^n \lambda_j U_j$, it means that for each $k$, column $k$ of $A$ can be expressed as $\sum_{j=1}^n \lambda_j u_k^{(j)}$, where $u_k^{(j)}$ is the column $k$ of matrix $U_j$.  If $u_k^{(j)}$ and $u_k^{(h)}$ are orthogonal for all $j \neq h$, then  the columns of matrix $T_k$ are orthogonal.  
\end{rem}





\section*{Acknowledgements}
The reviewer did a very careful reading of the manuscript and his comments have led to a great improvement of the article.
The  author acknowledges the financial support of a post-doctoral fellowship from the Pacific Institute for the Mathematical Sciences. The author is
grateful to John Benedetto for teaching him the subject of harmonic analysis and the theory of frames.  The author wishes to thank \"{O}zg\"{u}r Yilmaz for fruitful discussion about random matrices.


\bibliographystyle{amsplain}
\bibliography{2013RMbib}

\end{document}